\documentclass[12pt, reqno]{amsart}
\usepackage{amsmath, amsthm, amscd, amsfonts, amssymb, graphicx, color}
\usepackage[bookmarksnumbered, colorlinks, plainpages]{hyperref}
\hypersetup{colorlinks=true,linkcolor=red, anchorcolor=green, citecolor=cyan, urlcolor=red, filecolor=magenta, pdftoolbar=true}

\newcommand*{\R}{\ensuremath{\mathbb{R}}}

\newtheorem{thm}{Theorem}[section]
\newtheorem{cor}{Corollary}[section]
\newtheorem{lem}{Lemma}[section]
\newtheorem{rem}{Remark}[section]
\newtheorem{exem}{Example}[section]
\newtheorem{defin}{Definition}[section]
\newtheorem{pr}{Proposition}[section]

\numberwithin{equation}{section}

\begin{document}
\setcounter{page}{1}

\title[Composite Bernstein Cubature]{Composite Bernstein Cubature}

\author[A. M. Acu, H. Gonska]{Ana-Maria Acu$^1$$^{*}$  and Heiner Gonska$^2$}

\address{$^{1}$Lucian Blaga University of Sibiu, Department of Mathematics and
\newline
 Informatics, Str. Dr. I. Ratiu, No.5-7, RO-550012  Sibiu, Romania.}
\email{\textcolor[rgb]{0.00,0.00,0.84}{acuana77@yahoo.com}}

\address{$^{2}$ University of Duisburg-Essen, Faculty of Mathematics, Forsthausweg 2, 47057 Duisburg, Germany.}
\email{\textcolor[rgb]{0.00,0.00,0.84}{heiner.gonska@uni-due.de}}




\date{Received: xxxxxx; Revised: yyyyyy; Accepted: zzzzzz.
\newline \indent $^{*}$ Corresponding author}

\begin{abstract}
We consider a sequence of composite bivariate Bernstein operators
and the cubature formula associated with them. The upper-bounds for
the remainder term of cubature formula are described in terms of
moduli of continuity of order two. Also we include some results
showing how non-multiplicative the integration functional is.
\end{abstract} \maketitle

\section{Introduction}

We reconsider (composite) bivariate Bernstein approximation and the corres-\linebreak ponding cubature formulae.
This is motivated by a recent series of articles by Barbosu et al. (see \cite{BarPop2008}-\cite{BarMic2010a}).
However, some of these papers contain rather misleading statements and claims which can hardly be verified.
The present is written with the intention to clean up some of the bugs, to optimize and generalize certain
 estimates, and thus to further describe the situation at hand.

Our present contribution is a continuation of \cite{GonRas2013}. Historically the origin of the method discussed
seems to be in the article \cite{Vernescu} by D.D. Stancu and A. Vernescu.

\section{A general result}

 We first introduce some notation which will be needed to formulate the general result.
 \begin{defin} Let $I$ and $J$ be compact intervals of the real axis
 and let\linebreak  $L:C(I)\to C(I)$ and $M:C(J)\to C(J)$ be discretely defined
 operators, i.e.,
 $$ L(g;x)=\displaystyle\sum_{e\in E}g(x_e)A_{e}(x),\,g\in C(I),x\in I, $$
 where $E$ is a finite index set, the $x_{e}\in I$ are mutually
 distinct and $A_{e}\in C(I)$,\linebreak $ e\in E$.

 Analogously,
 $$ M(h;y)=\displaystyle\sum_{f\in F}h(y_{f})B_{f}(y),\,h\in C(J), y\in J. $$
 If $L$ is of the form above, then its parametric extension to
 $C(I\times J)$ is given by
 $$ _{x}L(F;x,y)=L(F_{y};x)=\displaystyle\sum_{e\in E}F_{y}(x_e)A_{e}(x)=\sum_{e\in E}F(x_e,y)A_e(x). $$
 Here $F_{y}, y\in J$, denote the partial functions of $F$ given by
 $F_{y}(x)=F(x,y),x\in I$.

  Similarly,
 $$ _{y}M(F;x,y)=\displaystyle\sum_{f\in F}F(x,y_{f})B_{f}(y).  $$
 The tensor product of $L$ and $M$ (or $M$ and $L$) is given by
 $$ \left(_{x}L\circ_{y}M\right)(F;x,y)=\displaystyle\sum_{e\in E}\sum_{f\in F}F(x_e,y_f)A_{e}(x)B_{f}(y). $$
 \end{defin}
 The theorem below is given in terms of so-called partial moduli of
 smoothness of order $r$, given for the compact intervals
 $I,J\subset \mathbb{R}$, for $F\in C(I\times J)$, $r\in \mathbb{N}_0$ and $\delta\in
 \mathbb{R}_{+}$ by
 $$ \omega_{r}(F;\delta,0):=\!\sup\left\{\left|\displaystyle\sum_{\nu=0}^r(-1)^{r-\nu}{r\choose \nu}F(x\!+\!\nu h,y)\right|:
  (x,y),(x\!+\!rh,y)\in\! I\!\times \!J, |h|\!\leq \!\delta\right\} $$
  and symmetrically by
$$ \omega_{r}(F;0,\delta):=\!\sup\left\{\left|\displaystyle\sum_{\nu=0}^r(-1)^{r-\nu}{r\choose \nu}F(x,y\!+\!\nu h)\right|:
  (x,y),(x,y\!+\!rh)\!\in\! I\!\times\! J, |h|\!\leq \!\delta\right\}. $$
 The total modulus of smoothness of order $r$ is defined by
$$ \omega_{r}(F;\delta_1,\delta_2):=\displaystyle\sup\left\{\left|\displaystyle\sum_{\nu=0}^r(-1)^{r-\nu}{r\choose \nu}F(x+\nu h_1,y+\nu h_2)\right|:\right.$$
$$\left.
  (x,y),(x+rh_1,y+rh_2)\in I\times J, |h_1|\leq \delta_1,\, |h_2|\leq \delta_2\right\}.  $$

We now formulate and prove a simplified form of Theorem 37 in [6].
\begin{thm}\label{t1}
Let $L$ and $M$ be discretely defined operators as given above such
that
$$  |(g-Lg)(x)|\leq\displaystyle\sum_{\rho=0}^r\Gamma_{\rho,L}(x)\omega_{\rho}(g;\Lambda_{\rho,L}(x)),\,g\in C(I), x\in I,  $$
and
$$  |(h-Mh)(y)|\leq\displaystyle\sum_{\sigma=0}^s\Gamma_{\sigma,M}(y)\omega_{\sigma}(h;\Lambda_{\sigma,M}(y)),\,h\in C(J), y\in J. $$
Here $\omega_{\rho},\rho=0,\dots,r$, denote the moduli of order
$\rho$, and $\Gamma$ and $\Lambda$ are bounded functions.
Analogously for $M$. Then for $(x,y)\in I\times J$ and $F\in
C(I\times J)$ the following hold:
\begin{align*} \left|\left[F-(_{x}L\circ _{y}M)F\right](x,y)  \right| & \leq \displaystyle \sum_{\rho=0}^{r}\Gamma_{\rho,L}(x)\omega_{\rho}(F;\Lambda_{\rho,L}(x),0)\\
&+\|L\|\sum_{\sigma=0}^s\Gamma_{\sigma,M}(y)\omega_{\sigma}(F;0,\Lambda_{\sigma,M}(y)),\end{align*}
 where $\|L\|$ denotes the operator norm of $L$, which is finite due
 to the form of $L$.
\end{thm}
\begin{proof}
We have
\begin{align*}
|\left[F-(_{x}L\circ _{y}M)F\right](x,y)|&=
|\left[(Id-_{x}L)+_{x}L\circ(Id-_{y}M)\right](F;x,y)|\\
&\leq|(Id-_{x}L)(F;x,y)|+|_xL\circ (Id-_yM)(F;x,y)|\\
&=:E_1(x,y)+E_2(x,y).
\end{align*}
Now, for $x\in I$,
\begin{align*}
E_1(x,y)&=|(Id-L)(F_y;x)|\leq\displaystyle\sum_{\rho=0}^r\Gamma_{\rho,L}(x)\cdot\omega_{\rho}\left(F_y;\Lambda_{\rho,L}(x)\right)\\
&\leq\displaystyle\sum_{\rho=0}^r\Gamma_{\rho,L}(x)\cdot\omega_{\rho}\left(F;\Lambda_{\rho,L}(x),0\right).
\end{align*}
Furthermore, with $G:=(Id-_yM)F$, we have
$$  E_2(x,y)=|_xL(G;x,y)|=|L(G_y;x)|\leq\|L(G_y)\|_{\infty,x\in I}. $$
Here again $G_y\in C(I)$ for all $y\in J$. By our assumption on $L$
we have for any $g\in C(I)$ that
$$  \|Lg\|_{\infty}\leq\left(1+\displaystyle\sum_{\rho=0}^r2^{\rho}\cdot\|\Gamma_{\rho,L}\|_{\infty}\right)\cdot\|g\|_{\infty}. $$
Hence $\|L\|<\infty$.

In the situation at hand we have
\begin{align*}
\|G_y\|_{\infty}&\!=\!\|\left[(Id\!-_yM)F\right]_y(\cdot)\|_{\infty}\!=\!\|(Id\!-_yM)F(\cdot,y)\|_{\infty}\!=\!\|(Id-_yM)F_x(y)\|_{\infty, x\in I}\\
&\leq\|\displaystyle\sum_{\sigma=0}^s\Gamma_{\sigma,M}(y)\cdot\omega_{\sigma}\left(F_x;\Lambda_{\sigma,M}(y)\right)\|_{\infty}
 \leq\displaystyle\sum_{\sigma=0}^s\Gamma_{\sigma,M}(y)\cdot\displaystyle\sup_{x\in I}\omega_{\sigma}\left(F_x;\Lambda_{\sigma,M}(y)\right)\\
&=\displaystyle\sum_{\sigma=0}^s\Gamma_{\sigma,M}(y)\cdot\omega_{\sigma}\left(F;0,\Lambda_{\sigma,M}(y)\right).
\end{align*}
Hence
\begin{align*}
E_1(x,y)+E_2(x,y)&\leq\displaystyle\sum_{\rho=0}^r\Gamma_{\rho,L}(x)\cdot\omega_{\rho}(F;\Lambda_{\rho,L}(x),0)\\
&+\|
L\|\cdot\displaystyle\sum_{\sigma=0}^s\Gamma_{\sigma,M}(y)\cdot\omega_{\sigma}(F;0,\Lambda_{\sigma,M}(y)).
\end{align*}
\end{proof}

\section{Application to bivariate Bernstein operators}
\begin{exem} \label{e1} If we take $L=B_{n_1}$ and $M=B_{n_2}$ with two
classical Bernstein operators mapping $C[0,1]$ into $C[0,1]$, then
for $F\in C([0,1]\times [0,1])$ and $(x,y)\in[0,1]\times [0,1]$
$$  \left(_{x}B_{n_1}\circ _{y}B_{n_2}\right)\left(F;x,y\right)
=\displaystyle\sum_{i_1=0}^{n_1}\sum_{i_2=0}^{n_2}F\left(\frac{i_1}{n_1},\frac{i_2}{n_2}\right)p_{n_1,i_1}(x)p_{n_2,i_2}(y),
$$
where $p_{n,i}(x)={i \choose n}x^i(1-x)^{n-i}, x\in[0,1],$ and
 \begin{align*}&\left|\left[F\!-\!\left(_{x}B_{n_1}\circ
_{y}B_{n_2}\right)F\right](x,y)\right|\!\leq\!
\displaystyle\frac{3}{2}\left[\omega_{2}\left(F;\sqrt{\frac{x(1\!-\!x)}{n_1}},0\right)\!+\!\omega_{2}\left(F;0,\sqrt{\frac{y(1\!-\!y)}{n_2}}\right)\right]\\
&\leq \displaystyle\frac{3}{2}\left[\|
F^{(2,0)}\|_{\infty}\frac{x(1-x)}{n_1}+\|
F^{(0,2)}\|_{\infty}\frac{y(1-y)}{n_2}\right],  F\in
C^{2,2}([0,1]\times [0,1]).
\end{align*}
\end{exem}
\begin{proof} We apply Theorem \ref{t1} with $r=s=2$,
$\Gamma_{0,B_n}=\Gamma_{1,B_n}=0$,
$\Gamma_{2,B_n}=\displaystyle\frac{3}{2}$,
$\Lambda_{2,B_{n}}(z)=\sqrt{\frac{z(1-z)}{n}}$, for
$n\in\{n_1,n_2\}$. The latter two choices are possible due to a
well-known result of P\u alt\u anea (see \cite{RP}) showing that for
the univariate Bernstein operators one has
$$ |f(x)-B_{n}(f,x)|\leq\displaystyle\frac{3}{2}\omega_{2}\left(f,\sqrt{\frac{x(1-x)}{n}}\right). $$
\end{proof}
\begin{rem}\label{r1} From the last inequality we get
$$\left|f(x)-B_{n}(f;x)\right|\leq \displaystyle\frac{3}{2}\| f^{\prime\prime}\|_{\infty}\displaystyle\frac{x(1-x)}{n},\, f\in C^2[0,1].$$
This is worse than the known inequality
$$ |f(x)-B_{n}(f;x)|\leq\displaystyle\frac{1}{2}\| f^{\prime\prime}\|_{\infty}\frac{x(1-x)}{n}.  $$
Our inequality was obtained from the more general statement in terms
of $\omega_{2}$ and well-known properties of the modulus.

However, we
can use instead Theorem 1 in \cite{Orman} (take $p=q=2$,
$p^{\prime}=q^{\prime}=0$, $r=s=0$,
$\Gamma_{0,0,B_{n_1}}(x)=\displaystyle\frac{1}{2}\cdot\frac{x(1-x)}{n_1}$
and
$\Gamma_{0,0,B_{n_2}}(y)=\displaystyle\frac{1}{2}\cdot\frac{y(1-y)}{n_2}$)
to arrive at
\begin{align*}
|\left[F-(_{x}B_{n_1}\circ_{y}B_{n_2})F\right](x,y)|&\leq\displaystyle\frac{1}{2}\frac{x(1-x)}{n_1}\|F^{(2,0)}\|_{\infty}+\displaystyle\frac{1}{2}\frac{y(1-y)}{n_2}
\|F^{(0,2)}\|_{\infty}\\
&+\displaystyle\frac{1}{4}\frac{x(1-x)y(1-y)}{n_1n_2}\|F^{(2,2)}\|_{\infty}\nonumber\\
&\leq\displaystyle\frac{1}{8n_1}\|F^{(2,0)}\|_{\infty}+\frac{1}{8n_2}\|F^{(0,2)}\|_{\infty}+\frac{1}{64n_1n_2}\|
F^{(2,2)}\|_{\infty}.
\end{align*}
An estimate of this kind can be found in Theorem 2.3 of
\cite{BarPop2008}.

Such three-term expressions typically appear if one writes (I
denoting the identity)
$$  I-A\circ B=I-A+I-B-(I-A)\circ (I-B)=(I-A)\oplus(I-B),$$
that is, if one uses the fact that the remainder of the tensor
product is the Boolean sum of the errors of the parametric extension.
The approach behind the above Theorem \ref{t1} invokes the
decomposition
$$ I-A\circ B=I-A+A\circ(I-B), $$
and therefore leads to the two-term bound.
\end{rem}

\section{The Bernstein type cubature formula revisited}

In this section we give a new upper bound for the approximation error
of cubature formula  associated with  the bivariate
Bernstein operators. The bounds are described in terms of
moduli of continuity of order two. The consideration of this
cubature formula  is motivated by B\u arbosu and Pop's result
\cite{BarPop2009}. It deems necessary to also correct some of the
wrong statements made there, in particular those with respect to
Boolean sums.

Integrating the bivariate Bernstein polynomials for $F\in
C([0,1]\times[0,1])$ one arrives at the following cubature formula
\begin{equation}\label{*}
\displaystyle\int_{0}^1\int_{0}^1F(x,y)dxdy=\displaystyle\frac{1}{(n_1+1)(n_2+1)}\sum_{i_1=0}^{n_1}\sum_{i_2=0}^{n_2}F\left(\frac{i_1}{n_1},\frac{i_2}{n_2}\right)
+R_{n_1,n_2}[F],
\end{equation}
where the remainder is bounded as follows:
$$ \left|R_{n_1,n_2}[F]\right|\leq \displaystyle\frac{1}{12n_1}\|F^{(2,0)}\|+\frac{1}{12n_2}\|F^{(0,2)}\|
+\frac{1}{144n_1n_2}\|F^{2,2}\|_{\infty},$$
 if  $F \in
C^{2,2}([0,1]\times[0,1]). $

 This follows from the three-term upper bound of Remark \ref{r1}.
 See \cite{BarPop2009} where the same integration error bound can be
 found.

 The two-term bound from Example \ref{e1} leads to the following
 \begin{thm} For the remainder term of the cubature formula (\ref{*}), $n_1,n_2\in
 \mathbb{N}$ and $F\in C([0,1]\times[0,1])$ there holds
 $$ \left|R_{n_1,n_2}[F]\right|\leq \displaystyle\frac{3}{2}\left[\int_{0}^1\omega_{2}\left(F;\sqrt{\frac{x(1-x)}{n_1}},0\right)dx
 +\int_{0}^1\omega_{2}\left(F;0,\sqrt{\frac{y(1-y)}{n_2}}\right)dy\right]. $$
 Moreover, if $F\in C^{2,2}([0,1]\times[0,1])$, then the above implies
 $$ |R_{n_1,n_2}[F]|\leq \displaystyle\frac{1}{4}\left(\frac{1}{n_1}\|F^{(2,0)}\|_{\infty}+\frac{1}{n_2}\|F^{(0,2)}\|_{\infty}\right). $$
 \end{thm}
 \begin{proof} All that needs to be observed is that a function of
 type $\left[0,1/2\right]\ni z \to \omega_{2}(F;z,0)$ (with F fixed and
 continuous) is continuous, thus integrable.  The mixed moduli of smoothness of order $(k,l)$, with $k,l\in
 \mathbb{N}_0$,
 given for $\delta_1,\delta_2\geq 0$ by
 \begin{align*}\omega_{k,l}(F;\delta_1,\delta_2)&:=\displaystyle\sup\left\{ \left|\displaystyle\sum_{\nu=0}^k\sum_{r=0}^{l}(-1)^{\nu+\mu}{k\choose \nu}{l\choose \mu}
 F(x+\nu\cdot h_1, y+\mu\cdot h_2)\right|: \right.\\
 & \left. (x,y),(x+kh_1,y+lh_2)\in[0,1]^2, |h_i|\leq\delta_i,i=1,2
 \right\},
 \end{align*}
 is a positive, continuous and non-decreasing function with respect
 to both variables (see \cite{habil}, \cite{rusa}).
 For continuous $F$ these moduli are continuous in $\delta_1$ and
 $\delta_2$ and satisfy
 $$
 \omega_{k}(F;\delta_1,0)=\omega_{k,0}(F;\delta_1,\delta_2)\textrm{
 and }
\omega_{k}(F;0,\delta_1)=\omega_{0,k}(F;\delta_1,\delta_2).
$$
The latter is only relevant to us for $k=2$.
 \end{proof}

\section{The composite bivariate Bernstein operators}

In this section we construct the bivariate composite Bernstein
operators and the order of convergence is considered involving the
second modulus of continuity. Also, some inequalities of
Tchebycheff-Gr\"{u}ss type will be proven. These results are
obtained using some general inequalities published in
\cite{AnaMaria}, \cite{teza}. In order to give the main results of
this section, we recall the following facts:
\begin{itemize}
\item[1.] For $a,b\in \mathbb{R}$, $a<b$, and $f\in
\mathbb{R}^{[a,b]}$ the Bernstein polynomial of degree $n\in
\mathbb{N}$ associated to $f$ is given for $x\in[a,b]$, by
$$ B_{n}^{[a,b]}(f;x)=\displaystyle\frac{1}{(b-a)^n}\displaystyle\sum_{i=0}^n{n\choose i}(x-a)^i(b-x)^{n-i}f\left(a+i\frac{b-a}{n}\right).$$
\item[2.] For $g\in C^2[a,b]$ one has
$$ g(x)-B_{n}^{[a,b]}(g;x)=-\displaystyle\frac{(x-a)(b-x)}{2n}g^{\prime\prime}(\xi_{x}),\,\xi_{x}\in(a,b). $$
\end{itemize}
If we divide $[0,1]$ into subintervals
$\left[\displaystyle\frac{k-1}{m},\frac{k}{m}\right]$,
$k=1,\dots,m\in \mathbb{N}$, then on
$\left[\displaystyle\frac{k-1}{m},\frac{k}{m}\right]$ we consider
\begin{align*}
B_{n,k}(f;x)&=B_{n}^{\left[\frac{k-1}{m},\frac{k}{m}\right]}(f;x)\\
&=m^{n}\sum_{i=0}^{n}{n\choose
i}\left(x-\frac{k-1}{m}\right)^i\left(\frac{k}{m}-x\right)^{n-i}f\left(\frac{kn-n+i}{nm}\right).
\end{align*}
Now we compose the $B_{n,k}$ to obtain the positive linear operator
$\overline{B}_{n,m}:\mathbb{R}^{[0,1]}\to C[0,1]$,
$$\overline{B}_{n,m}(f;x)=B_{n,k}(f;x),\textrm{ if } x\in
\left[\displaystyle\frac{k-1}{m},\frac{k}{m}\right],\, 1\leq k\leq
m.$$

From now on (subscripted) symbols $n...$ will refer to a polynomial degree. \linebreak (Subscripted) numbers $m...$ will be related to grids.
 Each function $\overline{B}_{n,m}(f)$ is a Schoenberg spline of degree $n$ with
 respect to the knot sequence given as\linebreak  follows:
 $$\begin{array}{rr}
 0=\displaystyle\frac{0}{m}& (n+1)-\textrm{fold}\\
 \vspace{-0.4cm}\\
 \displaystyle\frac{1}{m}& n-\textrm{fold}\\
  \vdots & \vdots\\
 \displaystyle\frac{m-1}{m}& n-\textrm{fold}\\
 \vspace{-0.4cm}\\
 1=\displaystyle\frac{m}{m}& (n+1)-\textrm{fold}
 \end{array}$$
 We renounce to give a precise numbering of the knots since this
 will not be needed below.
Thus $\overline{B}_{n,m}$ reproduces linear functions, interpolates
 at $\displaystyle\frac{k}{m}$, $0\leq k\leq m$ and has operator
 norm $\|\overline{B}_{n,m}\|=1$.

 For $n_1,n_2,m_1,m_2\in \mathbb{N}$ we now consider the parametric
 extension
 $_{x}\overline{B}_{n_1,m_1}$ and $_{y}\overline{B}_{n_2,m_2}$ and
 their product $_{x}\overline{B}_{n_1,m_1}\circ _{y}\overline{B}_{n_2,m_2}
 $. For brevity the latter will be denote by
 $\overline{\mathcal{B}}$.

 For $(x,y)\in
\left[\displaystyle\frac{k-1}{m_1},\frac{k}{m_1}\right]\times
\left[\displaystyle\frac{l-1}{m_2},\frac{l}{m_2}\right]$, it follows
\begin{align*}
\overline{\mathcal{B}}(f;x,y)&=m_1^{n_1}\cdot
m_2^{n_2}\displaystyle\sum_{i=0}^{n_1}\sum_{j=0}^{n_2}{n_1\choose
i}{n_2\choose j}
\left(x-\frac{k-1}{m_1}\right)^i\left(\frac{k}{m_1}-x\right)^{n_1-i}\\
&\cdot\left(y-\frac{l-1}{m_2}\right)^j
\left(\frac{l}{m_2}-y\right)^{n_2-j}
f\left(\frac{k-1}{m_1}+\frac{i}{m_1n_1},\frac{l-1}{m_2}+\frac{j}{n_2m_2}\right)
\end{align*}
and \begin{align*} \left|
f(x,y)\!-\!\overline{\mathcal{B}}(f;x,y)\right|
&\!=\frac{\left(x\!-\!\frac{k\!-\!1}{m}\right)\left(\frac{k}{m_1}\!-\!x\right)}{2n_1}
\|f^{(2,0)}\|_{\infty}\!+\!\frac{\left(y\!-\!\frac{l-1}{m_2}\right)\left(\frac{l}{m_2}\!-\!y\right)}{2n_2}\|f^{(0,2)}\|_{\infty}\\
&+\frac{\left(x-\frac{k-1}{m}\right)\left(\frac{k}{m_1}-x\right)\left(y-\frac{l-1}{m_2}\right)\left(\frac{l}{m_2}-y\right)}{4n_1n_2}
\|f^{(2,2)}\|_{\infty},\end{align*} where
 $f\in C^{2,2}([0,1]\times [0,1])$.

Using Theorem 1 again we get
\begin{thm}\label{t3} For $f\in C([0,1]\times[0,1])$, $n_1,n_2,m_1,m_2\in \mathbb{N}$ and
$(x,y)\in[0,1]\times[0,1]$ there holds
\begin{align*}\left| f(x,y)- \overline{\mathcal{B}}(f;x,y\right|&\leq \displaystyle\frac{3}{2}
 \left\{\omega_{2}\left(f;\sqrt{\frac{\left(x-\frac{k-1}{m_1}\right)\left(\frac{k}{m_1}-x\right)}{n_1}},0\right)\right.\\
 &+\left.\omega_{2}\left(f;0,\sqrt{\frac{\left(y-\frac{l-1}{m_2}\right)\left(\frac{l}{m_2}-y\right)}{n_2}}\right)\right\},\end{align*}
 if  $(x,y)\in \left[\frac{k-1}{m_1},\frac{k}{m_1}\right]\times
 \left[\frac{l-1}{m_2},\frac{l}{m_2}\right]$, $1\leq k\leq m_1$, $1\leq l\leq m_2$.
\end{thm}
\begin{proof} For the univariate case we have
$$ \left| \overline{B}_{n_1,m_1}(f;x)-f(x) \right|\leq \displaystyle\frac{3}{2}
\omega_{2}\left(f;\sqrt{\frac{\left(x-\frac{k-1}{m_1}\right)\left(\frac{k}{m_1}-x\right)}{n_1}}\right),
$$
for $x\in\left[\frac{k-1}{m_1},\frac{k}{m_1}\right]$, $1\leq k\leq
m_1$. Here $\omega_2$ is the second order modulus over $[0,1]$. An
analogous inequality holds for $\overline{B}_{n_2,m_2}$.

The theorem mentioned implies, with $r=s=2$, the inequality claimed.
\end{proof}
\begin{rem}As mentioned earlier, for $g\in C^{2}[a,b]$ one has
$$ |g(x)-B_{n}^{[a,b]}(g;x)|=\left|-\displaystyle\frac{(x-a)(b-x)}{2n}g^{\prime\prime}(\xi_{x})\right| \leq \displaystyle\frac{(b-a)^2}{8n}\|
g^{\prime\prime}\|_{[a,b],\infty}.$$ For
$[a,b]=\left[\frac{k-1}{m},\frac{k}{m}\right]$, the last expression
equals
$\displaystyle\frac{1}{8m^2n}\|g^{\prime\prime}\|_{\left[\frac{k-1}{m},\frac{k}{m}\right],\infty}$.

If $f\in C^{2,2}([0,1]\times[0,1])$ and $(x,y)\in[0,1]\times[0,1]$,
using Theorem 1 in \cite{Orman}, this leads to
 $$\left| f(x,y)- \overline{\mathcal{B}}(f;x,y\right|\leq
 \displaystyle\frac{1}{8m_1^2n_1}\|f^{(2,0)}\|_{\infty}+\frac{1}{8m_2^2n_2}\|f^{(0,2)}\|_{\infty}+\frac{1}{64m_1^2n_1m_2^2n_2}\|f^{(2,2)}\|_{\infty}. $$
For $m_1=m_2=1$ this is exactely the inequality in Remark \ref{r1}.
\end{rem}

\section{A Chebyshev-Gr\"{u}ss inequality}
In what follows we present an inequality for the bivariate composite
Bernstein operators, expressed in term of least concave majorant of
continuity. Let $C(X)$ be the Banach lattice of real valued
continuous functions defined on the compact metric space $(X,d)$.
\begin{defin} Let $f\in C(X)$. If, for $t\in[0,\infty)$, the
quantity
$$\omega_{d}(f;t):=\sup\left\{|f(x)-f(y)|, d(x,y)\leq t\right\}  $$
is the usual modulus of continuity, then its least concave mojorant
is given by
$$\tilde{\omega}_{d}(f,t)=\left\{\begin{array}{l}\displaystyle\sup_{0\leq x<t\leq y\leq d(X)}\frac{(t-x)\omega_{d}(f,y)+(y-t)
\omega_{d}(f,x)}{y-x},0\leq t\leq d(X),\\
\vspace{-0.4cm}\\
 \omega_{d}(f,d(X)), t>d(X), \end{array}\right. $$
 and $d(X)<\infty$ is the diameter of the compact space $X$.
 \end{defin}
Denote
$$ Lip_{r}=\left\{g\in C(X)\left|
|g|_{Lip_{r}}:=\displaystyle\sup_{d(x,y)>0}\frac{|g(x)-g(y)|}{
d^r(x,y)}<\infty\right. \right\}, \, 0<r\leq 1.
$$
$Lip_{r}$ is a dense subspace of $C(X)$ equipped with the supremum
norm $\|\cdot\|_{\infty}$ and $|\cdot|_{Lip_{r}}$ is a seminorm on
$Lip_{r}$.

The $K$-functional with respect to $(Lip_{r},|\cdot|_{Lip_{r}})$ is
given by
 $$ K(t,f;C(X),Lip_{r}):=\displaystyle\inf_{g\in Lip_{r}}\left\{
 \|f-g\|_{\infty}+t|g|_{Lip_{r}}\right\},\textrm{ for } f\in C(X)\textrm{ and } t\geq 0.  $$
 \begin{lem}\cite{m10} Every continuous function $f$ on $X$
 satisfies
 $$ K\left(\displaystyle\frac{t}{2},f; C(X),Lip_{1}\right)=\displaystyle\frac{1}{2}\tilde{\omega}_{d}(f,t), 0\leq t\leq d(X). $$
 \end{lem}
Let $H:C(X^2)\to C(X^2)$ be a positive linear
 operator reproducing constant function and define
 $$ T(f,g;x,y)=H(fg;x,y)-H(f;x,y)\cdot H(g;x,y). $$
 In order to give an inequality of  Chebyshev-Gr\"{u}ss type we
 recall a general result given by M. Rusu in \cite{teza}.
 \begin{thm}\label{t2.1}\cite{teza} If $f,g\in C(X^2)$ and $x,y\in X$ fixed, then
 the inequality
 $$ | T(f,g;x,y)\leq \displaystyle\frac{1}{4}\tilde{\omega}_{d}\left(f; 4\sqrt{H\left(d^2(\cdot,(x,y));x,y\right)}\right)
 \cdot \tilde{\omega}_{d}\left(g; 4\sqrt{H\left(d^2(\cdot,(x,y));x,y\right)}\right) $$
 holds, where $H\left(d^{2}(\cdot,(x,y));x,y\right)$ is the second
 moment of the bivariate operator $H$. We consider here the
 Euclidian metric $d_{2}$.
 \end{thm}
 \begin{pr}For $f,g\in C(X^2)$ and $x,y\in X$ fixed, the following
 Gr\"{u}ss type inequality holds
 \begin{align*}&|\overline{\mathcal{B}}(fg;x,y)-\overline{\mathcal{B}}(f;x,y)\cdot \overline{\mathcal{B}}(g;x,y)
 |\leq \displaystyle\frac{1}{4}\tilde{\omega}_{d_2}\left(f;
4\sqrt{\Psi(x,y)}\right)\cdot \tilde{\omega}_{d_2}\left(g;
 4\sqrt{\Psi(x,y)}\right)\\
 &\leq \displaystyle\frac{1}{4}\tilde{\omega}_{d_2}\left(f; 2\sqrt{\frac{1}{n_1m_1^2}+\frac{1}{n_2m_2^2}}\right)\cdot \tilde{\omega}_{d_2}\left(g;
 2\sqrt{\frac{1}{n_1m_1^2}+\frac{1}{n_2m_2^2}}\right)
 \end{align*}
 where $\Psi(x,y)=\displaystyle\frac{\left(x-\frac{k-1}{m_1}\right)\left(\frac{k}{m_1}-x\right)}{n_1}
  +\frac{\left(y-\frac{l-1}{m_2}\right)\left(\frac{l}{m_2}-y\right)}{n_2}$
  and $(x,y)\in \left[\frac{k-1}{m_1},\frac{k}{m_1}\right]\times \left[\frac{l-1}{m_2},\frac{l}{m_2}\right]$.
 \end{pr}

 \section{A cubature formula based on $\overline{\mathcal{B}}$}
In this section some upper-bounds of the error of cubature formula
associated with the bivariate Bernstein operators are given . In
\cite{Barbosu} D. B\u{a}rbosu, D.
 Micl\u{a}u\c{s} introduced the following cubature formula:
 $$\displaystyle\int_{0}^{1}\int_{0}^{1}f(x,y)dxdy=\sum_{k=1}^{m_1}\sum_{l=1}^{m_2}\int_{\frac{k-1}{m_1}}^{\frac{k}{m}}
 \int_{\frac{l-1}{m_2}}^{\frac{l}{m_2}}f(x,y)dxdy  $$
 $$ \approx \displaystyle\sum_{k=1}^{m_1}\sum_{l=1}^{m_2}\int_{\frac{k-1}{m_1}}^{\frac{k}{m}}
 \int_{\frac{l-1}{m_2}}^{\frac{l}{m_2}}\overline{\mathcal{B}}(f;x,y)dxdy
=\int_{0}^1\int_{0}^1\overline{\mathcal{B}}(f;x,y)dxdy:=\overline{\mathcal{I}}(f).
$$
 It follows
$$  \int_{\frac{k-1}{m_1}}^{\frac{k}{m}}
 \int_{\frac{l-1}{m_2}}^{\frac{l}{m_2}}\overline{\mathcal{B}}(f;x,y)dxdy$$  $$
 =m_1^{n_1}m_2^{n_2}\displaystyle\sum_{i=0}^{n_1}\sum_{j=0}^{n_2}{n_1\choose i}{n_2\choose j}\int_{\frac{k-1}{m_1}}^{\frac{k}{n_1}}
 \left(x-\frac{k-1}{m_1}\right)^{i}\left(\frac{k}{m_1}-x\right)^{n_1-i}dx$$
 $$ \cdot\int_{\frac{l-1}{m_2}}^{\frac{l}{m_2}}\left(y-\frac{l-1}{m_2}\right)^j\left(\frac{l}{m_2}-y\right)^{n_2-j}dy f
 \left(\frac{k-1}{m_1}+\frac{i}{m_1n_1},\frac{l-1}{m_2}+\frac{j}{n_2m_2}\right) $$
 $$=\displaystyle\sum_{i=0}^{n_1}\sum_{j=0}^{n_2}A_{n_1,n_2,m_1,m_2}f
 \left(\frac{k-1}{m_1}+\frac{i}{m_1n_1},\frac{l-1}{m_2}+\frac{j}{n_2m_2}\right),$$
 where $A_{n_1,n_2,m_1,m_2}=\displaystyle\frac{1}{m_1m_2(n_1+1)(n_2+1)}$.
 \begin{thm}For $f\in C^{2,2}\left([0,1]\times [0,1]\right)$ it
 follows
 \begin{align*} \left|\int_{0}^1\int_{0}^1f(x,y)dxdy-\overline{\mathcal{I}}(f) \right|&\leq
 \displaystyle\frac{1}{12n_1m_1^2}\|f^{(2,0)}\|_{\infty}+\frac{1}{12n_2m_2^2}\|f^{(0,2)}\|\\
&+\frac{1}{144n_1n_2m_1^2m_2^2}\|f^{(2,2)}\|_{\infty}.\end{align*}
 \end{thm}
 {\bf Proof.} We have
 \begin{align*}& \left|\int_{0}^1\int_{0}^1f(x,y)dxdy-\overline{\mathcal{I}}(f) \right| \\
 &=\left|\displaystyle\sum_{k=1}^{m_1}\sum_{l=1}^{m_2}\int_{\frac{k-1}{m_1}}^{\frac{k}{m_1}}\int_{\frac{l-1}{m_2}}^{\frac{l}{m_2}}f(x,y)dxdy-
 \displaystyle\sum_{k=1}^{m_1}\sum_{l=1}^{m_2}\int_{\frac{k-1}{m_1}}^{\frac{k}{m_1}}\int_{\frac{l-1}{m_2}}^{\frac{l}{m_2}}\overline{\mathcal{B}}(f;x,y)dxdy\right|\\
& \leq\displaystyle\sum_{k=1}^{m_1}\sum_{l=1}^{m_2}\int_{\frac{k-1}{m_1}}^{\frac{k}{m_1}}\int_{\frac{l-1}{m_2}}^{\frac{l}{m_2}}\left| f(x,y)-\overline{\mathcal{B}}(f;x,y)\right|dxdy \\
   &=
   \displaystyle\sum_{k=1}^{m_1}\sum_{l=1}^{m_{2}}\int_{\frac{k-1}{m_1}}^{\frac{k}{m_1}}\int_{\frac{l-1}{m_2}}^{\frac{l}{m_2}}\left[
   \frac{\left(x\!-\!\frac{k-1}{m_1}\right)\left(\frac{k}{m_1}\!-\!x\right)}{2n_1}\|f^{(2,0)}\|_{\infty}
   \!+\!\frac{\left(y-\frac{l-1}{m_2}\right)\left(\frac{l}{m_2}\!-\!y\right)}{2n_2}\|f^{(0,2)}\|_{\infty}\right.\\
   &+\left. \frac{\left(x-\frac{k-1}{m_1}\right)\left(\frac{k}{m_1}-x\right)\left(y-\frac{l-1}{m_2}\right)\left(\frac{l}{m_2}-y\right)}{4n_1n_2}
 \|  f^{(2,2)}\|_{\infty}\right]dxdy \\
&\leq \displaystyle\sum_{k=1}^{m_1}\sum_{l=1}^{m_2}\left[\frac{1}{12n_1m_1^3m_2}\|f^{(2,0)}\|_{\infty}+
\frac{1}{12n_2m_2^3m_1}\|f^{(0,2)}\|_{\infty}+\frac{1}{144n_1n_2m_1^3m_2^3}\|f^{(2,2)}\|_{\infty}\right]\\
&=\displaystyle\frac{1}{12n_1m_1^2}\|f^{(2,0)}\|_{\infty}+\frac{1}{12n_2m_2^2}\|f^{(0,2)}\|+\frac{1}{144n_1n_2m_1^2m_2^2}\|f^{(2,2)}\|_{\infty}. \end{align*}

One further estimate is given in
\begin{thm}For $f\in C^{2,2}\left([0,1]\times [0,1]\right)$ it
 follows
 $$ \left|\int_{0}^1\int_{0}^1f(x,y)dxdy-\overline{\mathcal{I}}(f) \right|\leq
 \displaystyle\frac{1}{4}\left\{\frac{1}{m_1^2n_{1}}\|f^{(2,0)}\|_{\infty}+\frac{1}{m_2^2n_2}\|f^{(0,2)}\|_{\infty}\right\}.
 $$
 \end{thm}
 \begin{proof} Integrating the error given in Theorem \ref{t3} leads
 to
 \begin{align*}&\left|\int_{0}^1\int_{0}^1f(x,y)dxdy-\overline{\mathcal{I}}(f) \right|\\
&\leq
 \displaystyle\frac{3}{2}\sum_{k=1}^{m_1}\sum_{l=1}^{m_2}\left\{
 \frac{1}{m_2}\int_{\frac{k-1}{m_1}}^{\frac{k}{m_1}}\omega_{2}\left(f;\sqrt{\frac{\left(x-\frac{k-1}{m_1}\right)\left(\frac{k}{m_1}-x\right)}{n_1}},0\right)dx\right.
 \\
 &+\left.\displaystyle\frac{1}{m_1}\int_{\frac{l-1}{m_2}}^{\frac{l}{m_2}}\omega_{2}\left(f;0,\sqrt{\frac{\left(y-\frac{l-1}{m_2}\right)
 \left(\frac{l}{m_2}-y\right)}{n_2}}\right)dy\right\}.
 \end{align*}
 Since $f\in C^{2,2}([0,1]\times [0,1])$ leads to
 \begin{align*} \left|\int_{0}^1\int_{0}^1f(x,y)dxdy\!-\!\overline{\mathcal{I}}(f) \right|&\leq
 \displaystyle\frac{3}{2}\sum_{k=1}^{m_1}\sum_{l=1}^{m_2}\left\{
 \frac{1}{m_2}\|f^{(2,0)}\|_{\infty}\int_{\frac{k-1}{m_1}}^{\frac{k}{m_1}}\frac{\left(x\!-\!\frac{k-1}{m_1}\right)\left(\frac{k}{m_1}\!-\!x\right)}{n_1}dx\right.
 \\
 &+\left.\displaystyle\frac{1}{m_1}\|f^{(0,2)}\|_{\infty}\int_{\frac{l-1}{m_2}}^{\frac{l}{m_2}}\frac{\left(y-\frac{l-1}{m_2}\right)
 \left(\frac{l}{m_2}-y\right)}{n_1}dy\right\}
 \\
 &=\displaystyle\frac{3}{2}\sum_{k=1}^{m_1}\sum_{l=1}^{m_2}\left\{
 \frac{1}{6m_1^3m_2n_1}\|f^{(2,0)}\|_{\infty}\!+\! \frac{1}{6m_1m_2^3n_2}\|f^{(0,2)}\|_{\infty}\right\} \\
 &= \displaystyle\frac{1}{4}\left\{\frac{1}{m_1^2n_{1}}\|f^{(2,0)}\|_{\infty}+\frac{1}{m_2^2n_2}\|f^{(0,2)}\|_{\infty}\right\}.
 \end{align*}
 \end{proof}

\section{Non-multiplicativity of the cubature formula}
In this section we will give some results which suggest how
non-multiplicative the functional
$\overline{\mathcal{I}}(f)=\displaystyle\int_{0}^1\int_{0}^1\overline{\mathcal{B}}(f;(x,y))dx
dy$ is.

Let $(X,d)$ be  a compact metric space and $L:C(X)\to \mathbb{R}$ be
a positive linear functional reproducing constant.  We consider the
positive bilinear functional
$$ D(f,g):=L(fg)-L(f)L(g). $$
\begin{thm}\label{t6.1} If $f,g\in C(X)$, $(X,d)$
 a compact metric space, then the inequality
$$|D(f,g)\leq\displaystyle\frac{1}{4}\tilde{\omega}_d\left(f;2\sqrt{L^2(d^2(\cdot,\cdot))}\right)\tilde{\omega}_d\left(g;2\sqrt{L^2(d^2(\cdot,\cdot))}\right)$$
holds.
\end{thm}
\begin{proof}Let $f,g\in C[a,b]$ and $r,s\in Lip_1$. Using the
Cauchy-Schwarz inequality for positive linear functional gives
$$|L(f)|\leq L(|f|)\leq\sqrt{L(f^2)\cdot L(1)}=\sqrt{L(f^2)},$$
so we have
$$ D(f,f)=L(f^2)-L(f)^2\geq 0. $$
Therefore, $D$ is a positive bilinear form on $C(X)$. Using the
Cauchy-Schwarz inequality for $D$ it follows
$$ |D(f,g)|\leq\sqrt{D(f,f)D(g,g)}\leq\|f\|_{\infty}\|g\|_{\infty}. $$
Since $L$ is a positive linear functional we can represent as
follows
$$ L(f):=\displaystyle\int_{X}f(t)d\mu(t), $$
where $\mu$ is a Borel probability measure on $X$, i.e.,
$\displaystyle\int_{X}d\mu(t)=1$. For $r\in Lip_1$, it follows
\begin{align*}
 D(r,r)&=L(r^2)-L(r)^2=\displaystyle\int_{X}r^2(t)d\mu(t)-\left(\int_{X}r(u)d\mu(u)\right)^2\\
 &=\displaystyle\int_{X}\left(r(t)-\int_{X}r(u)d\mu(u)\right)^2d\mu(t)=
 \displaystyle\int_{X}\left(\int_{X}(r(t)-r(u))d\mu(u)\right)^{2}d\mu(t)\\
 &\leq\displaystyle\int_{X}\left(\int_{X}(r(t)-r(u))^2d\mu(u)\right)d\mu(t)\\
 &\leq|r|_{Lip_1}^2\int_{X}\left(\int_{X}d^2(t,u)d\mu(u)\right)d\mu(t)\\
&=|r|_{Lip_1}^2L^t\left[L\left(d^2(t,\cdot)\right)\right]
 =|r|_{Lip_1}^2L^2\left(d^2(\cdot,\cdot)\right).
\end{align*}
For $r,s\in Lip_1$ we have
$$ |D(r,s)|\leq\sqrt{D(r,r)D(s,s)}\leq|r|_{Lip_1}|s|_{Lip_1}L^2\left(d(\cdot,\cdot)\right). $$
Moreover, for $f\in C(X)$ and $s\in Lip_1$, we have the estimate
$$ |D(f,s)|\leq\sqrt{D(f,f)D(s,s)}\leq \|f\|_{\infty}|s|_{Lip_1}\sqrt{L^2\left(d(\cdot,\cdot)\right)}. $$
In a similar way, if $r\in Lip_1$ and $g\in C(X)$, we have
$$ |D(r,g)|\leq\sqrt{D(r,r)D(g,g)}\leq \|g\|_{\infty}|r|_{Lip_1}\sqrt{L^2\left(d(\cdot,\cdot)\right)}. $$
Let $f,g\in C(X)$ be fixed and $r,s\in Lip_1$ arbitrary, then
\begin{align*}
|D(f,g)|&=|D(f-r+r,g-s+s)|\\
&\leq
|D(f-r,g-s)+|D(f-r,s)|+|D(r,g-s)|+|D(r,s))|\\
&\leq\|f-r\|_{\infty}\cdot\|g-s\|_{\infty}+\|f-r\|_{\infty}\cdot|s|_{Lip_1}\sqrt{L^2\left(d^2(\cdot,\cdot)\right)}\\
&+\|g-s\|_{\infty}\cdot|r|_{Lip_1}\sqrt{L^2\left(d^2(\cdot,\cdot)\right)}
+|r|_{Lip_1}|s|_{Lip_1}L^2\left(d^2(\cdot,\cdot)\right)\\
&=\left\{\|f-r\|_{\infty}+|r|_{Lip_1}\sqrt{L^2\left(d^2(\cdot,\cdot)\right)}\right\}
\left\{\|g-s\|_{\infty}+|s|_{Lip_1}\sqrt{L^2\left(d^2(\cdot,\cdot)\right)}\right\}.
\end{align*}
Passing to the infimum over $r$ and $s$, respectively, leads to
\begin{align*}
|D(f,g)|&\leq
K\left(\sqrt{L^2\left(d^2(\cdot,\cdot)\right)},f;C(X),Lip_1\right)\cdot
K\left(\sqrt{L^2\left(d^2(\cdot,\cdot)\right)},g;C(X),Lip_1\right)\\
&\leq\frac{1}{4}\tilde{\omega}\left(f;2\sqrt{L^2\left(d^2(\cdot,\cdot)\right)}\right)\tilde{\omega}\left(g;2\sqrt{L^2\left(d^2(\cdot,\cdot)\right)}\right)
\end{align*}
\end{proof}

Applying Theorem \ref{t6.1} for $L(f)=\overline{\mathcal{I}}(f)$ we
obtain the following result:
\begin{cor}\label{c6.1} If $f,g\in C([0,1]\times[0,1])$, then
\begin{align}\label{e6.1}\left|\overline{\mathcal{I}}(fg)-\overline{\mathcal{I}}(f)\overline{\mathcal{I}}(g)\right|\!&\leq\!
\displaystyle\frac{1}{4}\tilde{\omega}_{d_2}\left(f;2\sqrt{\frac{1}{3}\left(1\!+\!\frac{1}{n_1m_1^2}\!+\!\frac{1}{n_2m_2^2}\right)}\right)\\
&\cdot
\tilde{\omega}_{d_2}\left(g;2\sqrt{\frac{1}{3}\left(1\!+\!\frac{1}{n_1m_1^2}\!+\!\frac{1}{n_2m_2^2}\right)}\right)\nonumber
\end{align}
\end{cor}
\begin{proof}We have
\begin{align*}&\overline{\mathcal{I}}\left(d_{2}^{2}(\cdot,\cdot)\right)=\displaystyle\sum_{k,k_1=1}^{m_1}\sum_{l,l_1=1}^{m_2}\sum_{i,i_1=0}^{n_1}\sum_{j,j_1=0}^{n_2}
\frac{1}{m_1^2m_2^2(n_1+1)^2(n_2+1)^2}\\
&\cdot\left[\left(\displaystyle\frac{k_1-1}{m_1}+\frac{i_1}{m_1n_1}-\frac{k-1}{m_1}-\frac{i}{m_1n_1}\right)^2+
\left(\frac{l_1-1}{m_2}+\frac{j_1}{n_2m_2}-\frac{l-1}{m_2}-\frac{j}{n_2m_2}\right)^2\right]\\
&=\displaystyle\frac{1}{m_1^2(n_1\!+\!1)^2}\sum_{k,k_1=1}^{m_1}\!\!\sum_{i,i_1=0}^{n_1}\!\!\left(\frac{k_1\!-\!k}{m_1}\!+\!\frac{i_1\!-\!i}{m_1n_1}\right)^2
\!+\!\displaystyle\frac{1}{m_2^2(n_2\!+\!1)^2}\sum_{l,l_1=1}^{m_2}\sum_{j,j_1=0}^{n_2}\left(\frac{l_1\!-\!l}{m_2}\!+\!\frac{j_1\!-\!j}{m_2n_2}\right)^2\\
&=\displaystyle\frac{1}{3}\left(1+\frac{1}{m_1^2n_1}+\frac{1}{m_2^2n_2}\right).
\end{align*}
Therefore, using Theorem \ref{t6.1} it follows
\begin{align*}\left|\overline{\mathcal{I}}(fg)-\overline{\mathcal{I}}(f)\overline{\mathcal{I}}(g)\right|&\leq
\displaystyle\frac{1}{4}\tilde{\omega}_{d_2}\left(f;2\sqrt{\frac{1}{3}\left(1+\frac{1}{n_1m_1^2}+\frac{1}{n_2m_2^2}\right)}\right)\\
&\cdot
\tilde{\omega}_{d_2}\left(g;2\sqrt{\frac{1}{3}\left(1+\frac{1}{n_1m_1^2}+\frac{1}{n_2m_2^2}\right)}\right).
\end{align*}
\end{proof}
In the following part of this section we will give a
Chebyshev-Gr\"{u}ss type\linebreak  inequality which involves oscillations of
function. This result is obtained\linebreak  using a general inequality
published in \cite{AnaMaria}.
 Let
$Y$ be an arbitrary set and $B(Y^2)$ the set of all real-valued,
bounded   functions on $Y^2$. Take $a_n, b_n\in \R,\ n\geq 0$, such
that $\displaystyle\sum^{\infty}_{n=0}{\left|a_n\right|}<\infty$,
$\displaystyle\sum^{\infty}_{n=0}{a_n}=1$ and
$\displaystyle\sum^{\infty}_{n=0}{\left|b_n\right|}<\infty$,
$\displaystyle\sum^{\infty}_{n=0}{b_n}=1$, respectively.
Furthermore, let $x_n\in Y, n\geq 0$ and $y_m\in Y,\ m\geq 0$ be
arbitrary mutually distinct points. For $f\in B(Y^2)$ set
$f_{n,m}:=f(x_n,y_m)$. Now consider the functional $L:B(Y^2)\to \R$,
$Lf=\displaystyle\sum^{\infty}_{n=0}\displaystyle\sum^{\infty}_{m=0}{a_nb_mf_{n,m}}$.
The functional $L$ is linear and reproduces constant functions.

\begin{thm}\label{th1.1}\cite{AnaMaria}
The Chebyshev-Gr\"{u}ss-type inequality for the above linear
 functional $L$ is given by:
    \begin{equation*}
        \left|L(fg)-L(f)\cdot L(g)\right| \leq \frac{1}{2}\cdot osc_L(f)\cdot osc_L(g)\cdot \sum^{\infty}_{n,m,i,j=0,\ (n,m)\neq (i,j)}{\left|a_nb_ma_ib_j\right|},
    \end{equation*}
where $f,g\in B(Y^2)$ and we define the oscillations to be:
    \begin{align*}
        osc_L(f)&:=\sup\{\left|f_{n,m}-f_{i,j}\right|: n,m,i,j\geq
        0\}.
    \end{align*}
\end{thm}

\begin{thm}\label{th2}\cite{AnaMaria}
In particular, if $a_n\geq 0,\ b_m\geq 0$, $n,m\geq 0$, then $L$ is
a positive linear functional and we have:
    \begin{equation*}
        \left|L(fg)-L(f)\cdot L(g)\right| \leq \frac{1}{2}\cdot \left(1-\sum^{\infty}_{n=0}{a^2_n}\cdot \sum^{\infty}_{m=0}{b^2_m}\right)\cdot osc_L(f)\cdot osc_L(g),
    \end{equation*}
for $f,g\in B(Y^2)$ and the oscillations given as above.
\end{thm}

The following result gives us the non-multiplicative of the
functional ${\mathcal{I}}$ using discrete oscillations. This result
is better than (\ref{e6.1}) in the sense that the oscillations of
functions are relative only to certain points, while in (\ref{e6.1})
the oscillations, expressed in terms of $\tilde{\omega}$, are
relative to the whole interval $[0,1]$.

\begin{cor}\label{c6.1} If $f,g\in B([0,1]^2)$, then
$$\left|\overline{\mathcal{I}}(fg)-\overline{\mathcal{I}}(f)\overline{\mathcal{I}}(g)\right|\leq
\displaystyle\frac{1}{2}\left(1-\frac{1}{m_1m_2(n_1+1)(n_2+1)}\right)osc(f)osc(g).
$$
\end{cor}

\bibliographystyle{amsplain}

\begin{thebibliography}{99}

\bibitem{AnaMaria} A.M. Acu, M.D. Rusu, {\it New results concerning
Chebyshev-Gr\"{u}ss-type inequalities via discrete oscillations},
Applied Mathematics and Computation, {\bf 243}, 2014, 585-593.
\bibitem{BarPop2008} D. B\u arbosu, O. Pop, {\it On the Bernstein bivariate
approximation formula}, Carpathian J. Math., {\bf 24}(3), 2008,
 293-298.
\bibitem{BarPop2009} D. B\u arbosu, O.T. Pop, {\it A note on the
Bernstein's cubature formula}, Gen. Math. {\bf 17}(3), 2009,
161-172.
\bibitem{Barbosu} D. B\u{a}rbosu, D. Micl\u{a}u\c{s}, {\it On the
composite Bernstein type cubature formula}, General \linebreak Mathematics,
{\bf 18}(3), 2010, 73-81.
\bibitem{BarMic2010a} D. B\u arbosu, D. Micl\u au\c s, {\it On the
composite Bernstein type quadrature formula}, Rev. Anal.
Num\'{e}r.Th\'{e}or. Approx. {\bf 39}(1), 2010, 3-7.
\bibitem{BGKT} L. Beutel, H.H. Gonska, D. Kacs\'{o}, G. Tachev, {\it
On variation-diminishing Schoenberg operators: new quantitative
statements}, In: Multivariate Approximation and Interpolation with
Applications (ed. by M. Gasca), Monogr. Academia Ciencias de
Zaragoza {\bf 20}, 2002, 9-58.
\bibitem{Orman} L. Beutel, H. Gonska, {\it Quantitative Inheritance
Properties for Simultaneous Approximation by Tensor Product
Operators II: Applications}, In: Mathematics and its Applications
(Proc. 17th Scientific Session; ed. by G.V.Orman), 1-28. Brasov:
Editura Universitatii "Transilvania" 2003.
\bibitem{GonRas2013} H. Gonska, I. Rasa, {\it Sur la suite des
op\'{e}rateurs Bernstein compos\'{e}s}, Revue d'Analyse \linebreak  Numerique et
de Theorie de l' Approximation, {\bf 42} (2), 2013, 151-160.
 \bibitem{habil} H.
Gonska, {\it Quantitative Approximation in C(X)},
Habilitationsschrift, University of\linebreak  Duisburg, 1986.
\bibitem{m10} B.S. Mitjagin, E.M.Semenov, {\it Lack of interpolation
of linear operators in spaces of smooth functions}, Math. USSR-Izv.,
{\bf 11}, 1977, 1229-1266.
\bibitem{RP} R. P\u alt\u anea, {\it On some constants in approximation by
Bernstein operators}, General Mathematics, {\bf 16}(4), 2008,
137-148.

\bibitem{teza} M.D. Rusu, {\it Chebyshev-Gr\"{u}ss- and Ostrowski-type
Inequalities}, PhD Thesis, Duisburg-Essen University, 2014.
\bibitem{Vernescu} D.D. Stancu, A. Vernescu, {\it On some remarkable
positive polynomial operators of approximation},Revue d'Analyse
Numerique et de Theorie de l' Approximation, {\bf 28}(1), 1999,
85-95.
 \bibitem{rusa}
V. Ya. Yanchak, {\it The expansion of functions of a mixed class in
a series of algebraic polynomials}, Trudy Mat. Inst. Steklov.,
{\bf 128}, 1972, 242-256.


\end{thebibliography}

\end{document}